\documentclass{article}
\usepackage{amssymb}
\usepackage{amsfonts}
\usepackage{amsmath}

\newtheorem{theorem}{Theorem}

\newtheorem{corollary}{Corollary}

\newtheorem{lemma}{Lemma}

\newtheorem{remark}{Remark}

\begin{document}

\title{Population size of critical Galton-Watson processes under small
deviations and infinite variance\thanks{%
This work was supported by the Russian Science Foundation under grant
no.24-11-00037 https://rscf.ru/en/project/24-11-00037/} }
\author{Vladimir Vatutin\thanks{%
Steklov Mathematical Institute of Russian Academy of Sciences, 8 Gubkina
St., Moscow 119991 Russia Email: vatutin@mi-ras.ru},\thinspace\ Elena
Dyakonova\thanks{%
Steklov Mathematical Institute of Russian Academy of Sciences, 8 Gubkina
St., Moscow 119991 Russia Email: elena@mi-ras.ru}}

\date{\empty}

\maketitle

\begin{abstract}
We study the evolution of the population size distribution of a critical
Galton-Watson process with infinite variance of the offspring size of
particles assuming that the population size is unusually small at the
distant moment $n$ of observation.



\end{abstract}

\section{Introduction}

We consider a Galton-Watson branching process $\mathcal{Z}=\left\{
Z(n),n\geq 0\right\} $ starting at moment 0. Let $f(s):=\mathbf{E}s^{\xi }$ be the probability generating
function of the offspring number $\xi $ of one particle. We assume that the
reproduction law of particles satisfies the condition
\begin{equation}
\mathbf{E}\xi =1,\text{ }f(s)=s+\left( 1-s\right) ^{1+\alpha }L(1-s),\quad
0\leq s\leq 1,  \label{MainAssump}
\end{equation}%
where $\alpha \in (0,1]$ and $L(z)$ is a slowly varying function as $%
z\downarrow 0$. Denote $f_{0}(s):=s$ and introduce iterations of $f(s)$ by
the equalities
\begin{equation*}
f_{n}(s)=f(f_{n-1}(s)),\ n=1,2,....
\end{equation*}%
It is known \cite{Sl1968} that, given (\ref{MainAssump})%
\begin{equation}
(1-f_{n}(0))^{\alpha }L\left( 1-f_{n}(0)\right) \sim \frac{1}{\alpha n}
\label{SimpeForm}
\end{equation}%
as $n\rightarrow \infty $ and, therefore (see, for instance, \cite[Section
1.5, point $5^{0}$]{Sen76} or \cite{Sen67})
\begin{equation}
Q(n):=\mathbf{P}\left( Z(n)>0|Z(0)=1\right) =1-f_{n}(0)\sim \frac{L^{\ast
}(n)}{n^{1/\alpha }}\text{ \ \ as \ \ }n\rightarrow \infty ,
\label{SurvivalProbab}
\end{equation}%
where $L^{\ast }(n)$ is a function slowly varying at infinity. Moreover, for
any $\lambda >0$%
\begin{eqnarray}
\lim_{n\rightarrow \infty }\mathbf{E}\left[ e^{-\lambda
(1-f_{n}(0))Z(n)}|Z(n)>0,Z(0)=1\right] &=&\int_{0}^{\infty }e^{-\lambda
x}dM\left( x\right)  \notag \\
&=&1-\frac{1}{\left( 1+\lambda ^{-\alpha }\right) ^{1/\alpha }},
\label{Yaglom000}
\end{eqnarray}%
where $M\left( x\right) $ is a proper nondegenerate distribution such that $%
M(x)>0$ for any $x>0$.

It follows from (\ref{Yaglom000}) that, for any $x>0$
\begin{equation*}
\mathbf{P}\left( 0<(1-f_{n}(0))Z(n)\leq x|Z(0)=1\right) \approx M(x)\mathbf{P%
}\left( Z(n)>0|Z(0)=1\right)
\end{equation*}%
as $n\rightarrow \infty $. Thus, the typical size of the population survived
up to a distant moment $n$ is of order $(1-f_{n}(0))^{-1}$.

In the sequel we assume (if otherwise is not stated) that $Z(0)=1$ and study
the asymptotic behavior of the population size of the critical Galton-Watson
branching process given the event
\begin{equation*}
\mathcal{H}(l,r):=\big\{0<(1-f_{r}(0))Z(l)\leq 1\big\}  \label{Def_H}
\end{equation*}%
for various assumptions concerning the ratio between the parameters $l$ and $%
r$. We mainly deal with the case%
\begin{equation}
\mathcal{H}(n,\varphi (n)):=\big\{0<(1-f_{\varphi (n)}(0))Z(n)\leq 1\big\}
\label{Def_H1}
\end{equation}%
where $\varphi (n),\ n=1,2,...,$ is a deterministic function such that
\begin{equation}
\varphi (n)\rightarrow \infty \text{ and }\frac{\varphi (n)}{n}\rightarrow 0
\label{Def_phi}
\end{equation}%
as $n\rightarrow \infty $.

Our main result looks as follows.

\begin{theorem}
\label{T_Population_size} Let conditions (\ref{MainAssump}) be valid for $%
\alpha \in (0,1)$ and assumption (\ref{Def_phi}) hold true. Then, for any $%
\lambda >0$

1) if $m\rightarrow \infty $ and $m=o(n)$ then%
\begin{equation*}
\lim_{n\rightarrow \infty }\mathbf{E}\left[ e^{-\lambda (1-f_{m}(0))Z(m)}|%
\mathcal{H}(n,\varphi (n))\right] =\frac{1}{\left( 1+\lambda ^{\alpha
}\right) ^{1/\alpha +1}};
\end{equation*}

2) if $m\sim \theta n$ for some $\theta \in (0,1)$ then%
\begin{equation*}
\lim_{n\rightarrow \infty }\mathbf{E}\left[ e^{-\lambda (1-f_{m}(0))Z(m)}|%
\mathcal{H}(n,\varphi (n))\right] =\frac{1}{\left( 1-\theta +\left( \lambda
\left( 1-\theta \right) ^{1/\alpha }+\theta ^{1/\alpha }\right) ^{\alpha
}\right) ^{1/\alpha +1}};
\end{equation*}

3) if $m=n-\psi (n),$ where $n\gg \psi (n)\gg \varphi (n),$ then%
\begin{equation*}
\lim_{n\rightarrow \infty }\mathbf{E}\left[ e^{-\lambda (1-f_{\psi
(n)}(0))Z(m)}|\mathcal{H}(n,\varphi (n))\right] =\frac{1}{\left( 1+\lambda
\right) ^{\alpha +1}};
\end{equation*}

4) if $m=n-y\varphi (n),$ where $y\in (0,\infty ),$ then%
\begin{equation*}
\lim_{n\rightarrow \infty }\mathbf{E}\left[ e^{-\lambda (1-f_{n-m}(0))Z(m)}|%
\mathcal{H}(n,\varphi (n))\right] =\sum_{j=1}^{\infty }\frac{\alpha \Gamma
\left( j+\alpha \right) }{j!\left( 1+\lambda \right) ^{\alpha +j}}yM^{\ast
j}(y^{-1/\alpha }),
\end{equation*}%
where $\Gamma (\cdot )$ is the standard Gamma-function and $M^{\ast j}(x)$
is the $j$-th convolution of $M(x)$ with itself;

5) if $m=n-\chi (n)$ where $n\gg \varphi (n)\gg \chi (n)\geq 0$ then%
\begin{equation*}
\lim_{m\rightarrow \infty }\mathbf{E}\left[ e^{-\lambda \left( 1-f_{\varphi
(n)}(0)\right) Z(m)}|\mathcal{H}(n,\varphi (n))\right] =\alpha
\int_{0}^{1}x^{\alpha -1}e^{-\lambda x}dx.
\end{equation*}

In particular, if $T\rightarrow \infty $ and $T=o\left(
(1-f_{n}(0))^{-1}\right) $ then
\begin{equation*}
\lim_{m\rightarrow \infty }\mathbf{P}\left( Z(n)<xT|0<Z(n)\leq T\right)
=x^{\alpha },\quad 0\leq x\leq 1.
\end{equation*}
\end{theorem}

For the first sight condition (\ref{Def_H1}) looks rather restrictive and
unnatural. Let us explain that this is not the case. We know that if $Z(n)>0$
then the population size of\ a critical Galton-Watson process at a distant
moment $n$ is of order $1/\left( 1-f_{n}(0)\right) $ according to the now
classical Yaglom-Slack result (\ref{Yaglom000}). Thus, if we would like to
consider a nontypically small population size at moment $n,$ given $Z(n)>0$
we need to impose the assumption%
\begin{equation}
0<(1-f_{\varphi (n)}(0))Z(n)\leq w,  \label{CondGeneral}
\end{equation}%
where $w>0$ and $\varphi (n)=o(n)$. In view of (\ref{SurvivalProbab})\ and
properties of slowly varying functions we have, for any $w>0$%
\begin{equation*}
\frac{1-f_{\varphi (n)}(0)}{w}\sim \frac{L^{\ast }(\varphi (n))}{w\varphi
^{1/\alpha }(n)}\sim \frac{L^{\ast }(w^{\alpha }\varphi (n))}{\left(
w^{\alpha }\varphi (n)\right) ^{1/\alpha }}\sim 1-f_{w^{\alpha }\varphi
(n)}(0).
\end{equation*}%
Therefore, we may replace condition (\ref{CondGeneral}) by the condition%
\begin{equation*}
0<(1-f_{w^{\alpha }\varphi (n)}(0))Z(n)\leq 1.
\end{equation*}%
Now observe that the replacement $\varphi (n)\rightarrow w^{\alpha }\varphi
(n)$ keeps the validity of the assumption (\ref{Def_phi}) and, as a result,
the validity of all conclusions of Theorem \ref{T_Population_size}.

 We do not mention in Theorem \ref{T_Population_size} the
case $\alpha =1.$ See the next section for additional commentaries concerning this case.

\section{\protect\bigskip Auxiliary results}

The proof of Theorem \ref{T_Population_size} is based on the following
statements established in \cite{VDK2025}.

\begin{theorem}
\label{T_deviation} Let condition (\ref{MainAssump}) be valid for $\alpha
\in (0,1)$ and $\varphi (n),\ n=1,2,...,$ be a deterministic function
satisfying assumption (\ref{Def_phi}). Then
\begin{equation}
\mathbf{P}\left( \mathcal{H}(n,\varphi (n))\right) \sim \frac{1-f_{n}(0)}{%
\alpha n}\frac{1}{\Gamma \left( 1+\alpha \right) }\varphi (n).
\label{LargeDeviation}
\end{equation}
\end{theorem}

\begin{remark}
\label{Rem1} Setting $T=\left( 1-f_{\varphi (n)}(0)\right) ^{-1}$ and
observing that
\begin{equation*}
\frac{L\left( 1/T\right) }{T^{\alpha }}=\left( 1-f_{\varphi (n)}(0)\right)
^{\alpha }L\left( 1-f_{\varphi (n)}(0)\right) \sim \frac{1}{\alpha \varphi
(n)}  \label{Repres1}
\end{equation*}%
as $n\rightarrow \infty $ by (\ref{SimpeForm}), we may rewrite (\ref%
{LargeDeviation}) as
\begin{equation}
\mathbf{P}\left( 0<Z(n)\leq T\right) \sim \frac{1-f_{n}(0)}{\alpha n}\frac{1%
}{\alpha \Gamma \left( 1+\alpha \right) }\frac{T^{\alpha }}{L\left(
1/T\right) }.  \label{ModifDevoation}
\end{equation}
\end{remark}

\begin{remark}
\label{Rem2} If
\begin{equation}
f^{\prime }(1)=1,\quad \sigma ^{2}=f^{\prime \prime }(1)\in (0,\infty ),
\label{FinVariance}
\end{equation}%
then $L(x)\to \sigma ^{2}/2$ as $x\downarrow 0$ and (see, for instance,
\cite{AN72}, Ch.1, Sec. 9)
\begin{equation*}
1-f_{n}(0)\sim \frac{2}{\sigma ^{2}n},\quad n\rightarrow \infty .
\end{equation*}%
Hence it follows that if $n\gg \varphi (n)\rightarrow \infty $ then (\ref%
{LargeDeviation}) is replaced by
\begin{equation*}
\mathbf{P}\left( 0<Z(n)<T\right) \sim \frac{4}{\sigma ^{4}n^{2}}T,
\end{equation*}%
and is an evident corollary of the main result in \cite{NW2006}.

If, however, $\alpha=1$ and $\sigma^2=\infty$, then an asymptotic representation for $\mathbf{P}\left( 0<Z(n)\leq T\right)$ similar to (\ref{ModifDevoation}) is now know only for the case
\begin{equation}
\mathbf{P}\left( \xi \geq j\right) =\frac{\bar{L}(1/j)}{j^{2}},
\label{Alpha_1}
\end{equation}%
where $\bar{L}(z)$ is a function slowly varying at zero (see Remark 2 in \cite{VDK2025}).
 Since in the arguments to follow we use only the form of the right-hand side of (\ref{ModifDevoation}),
  one may check that all the statements of Theorem~\ref{T_Population_size} remain valid if either the condition (\ref{FinVariance}) or the conditin (\ref{Alpha_1}) is in force.
\end{remark}

We fix $n\in \left\{ 1,2,...\right\} $ and introduce the so-called reduced
Galton-Watson process $\{Z(m,n),0\leq m\leq n\}$, where $Z(m,n)$ is the
number of such generation $m$ particles in the original process $\mathcal{Z}$
that have nonempty offspring in generation $n$ (with the agreement that $%
Z(n,n)=Z(n))$.

\begin{theorem}
\label{T_reduced}(see \cite{VDK2025}) Let condition (\ref{MainAssump}) be
valid for $\alpha \in (0,1)$ and assumption (\ref{Def_phi}) hold true. Then,
for any $y\in (0,\infty )$ and any $j\geq 1$%
\begin{equation}
\lim_{n\rightarrow \infty }\mathbf{P}\left( Z(n-y\varphi (n),n)=j|\mathcal{H}%
(n,\varphi (n))\right) =\frac{\alpha \Gamma \left( j+\alpha \right) }{j!}%
yM^{\ast j}(y^{-1/\alpha }).  \label{InfVar}
\end{equation}
\end{theorem}

Let
\begin{equation*}
\beta (n):=\max \left( 0\leq m<n:Z(m,n)=1\right)
\end{equation*}%
be the birth moment of the most recent common ancestor (MRCA) of all
particles existing in the population at moment $n$ and let $d(n):=n-\beta
(n) $ be the distance from the moment $n$ of observation to the birth moment
of the MRCA. The following statement describing the distribution of the
random variable $d(n)$ is an evident corollary of Theorem \ref{T_reduced}.
This claim plays an important role in our subsequent arguments.

\begin{corollary}
\label{Cor1}Under the conditions of Theorem \ref{T_reduced}, for any $y\in
(0,\infty )$%
\begin{eqnarray*}
&&\lim_{n\rightarrow\infty}\mathbf{P}\left( Z(n-y\varphi
(n),n)=1|\mathcal{H}(n,\varphi (n))\right)\\
&&\qquad\qquad\quad=\lim_{n\rightarrow \infty }\mathbf{P}\left(
d(n)\leq y\varphi (n)|\mathcal{H}(n,\varphi (n))\right)
=\alpha\Gamma (1+\alpha) yM(y^{-1/\alpha }).
\end{eqnarray*}
\end{corollary}

It follows from (\ref{Yaglom000}) that%
\begin{equation*}
\int_{0}^{\infty }e^{-\lambda y}dM\left( x\right) =1-\frac{1}{\left(
1+\lambda ^{-\alpha }\right) ^{1/\alpha }}\sim \frac{1}{\alpha }\frac{1}{%
\lambda ^{a}},\
\end{equation*}%
as $\lambda \rightarrow \infty $. Hence we conclude by a Tauberian theorem
(see, for instance, \cite[Ch.XIII, Sec.5, Theorems 2 and 3]{Fel}) that%
\begin{equation*}
M\left( x\right) \sim \frac{1}{\alpha \Gamma \left( 1+\alpha \right) }%
x^{\alpha }
\end{equation*}%
as $x\rightarrow 0$. Thus,
\begin{equation}
\lim_{y\rightarrow \infty }\alpha \Gamma \left( 1+\alpha \right)
yM(y^{-1/\alpha })=1.  \label{LOcal}
\end{equation}%
This fact combined with Corollary \ref{Cor1} shows that, given $\mathcal{H}%
(n,\varphi (n))$ the distance to the MRCA of all individuals of the $n$-th
generation is of order $\varphi (n)$. We will use this result later on.

We need one more statement related to the MRCA. Namely, we show that the
limiting distribution specified by the right-hand side of (\ref{InfVar}) is
proper.

\begin{lemma}
\label{L_proper} For any $y>0$%
\begin{equation*}
\sum_{j=1}^{\infty }\frac{\alpha \Gamma \left( j+\alpha \right) }{j!}%
yM^{\ast j}(y^{-1/\alpha })=1.
\end{equation*}
\end{lemma}

\textbf{Proof}. The statement of the lemma will be proved if we show that,
for any $x>0$

\begin{equation}
U(x):=\sum_{j=1}^{\infty }\frac{\alpha \Gamma \left( j+\alpha \right) }{j!}%
M^{\ast j}(x)=x^{\alpha },  \label{Term0}
\end{equation}%
or, since $U(\cdot )$ is a measure, that (see \cite{Fel}, Ch. XIII, Sec.1,
Theorem 1a), for any $\lambda >0$%
\begin{equation}
\sum_{j=1}^{\infty }\frac{\alpha \Gamma \left( j+\alpha \right) }{j!}%
\int_{0}^{\infty }e^{-\lambda x}dM^{\ast j}(x)=\int_{0}^{\infty }e^{-\lambda
x}dx^{\alpha }=\frac{\Gamma \left( \alpha +1\right) }{\lambda ^{\alpha }}.
\label{Term1}
\end{equation}%
It follows from (\ref{Yaglom000}) that
\begin{equation*}
\int_{0}^{\infty }e^{-\lambda x}dM^{\ast j}(x)=\left( 1-\frac{\lambda }{%
\left( 1+\lambda ^{\alpha }\right) ^{1/\alpha }}\right) ^{j}.
\end{equation*}%
Now we set%
\begin{equation}
1-\frac{\lambda }{\left( 1+\lambda ^{\alpha }\right) ^{1/\alpha }}=t.
\label{Def_t}
\end{equation}%
It is not difficult to check that%
\begin{equation}
\sum_{j=1}^{\infty }\frac{\alpha \Gamma \left( j+\alpha \right) }{j!}%
t^{j}=\Gamma \left( \alpha +1\right) \left( \frac{1}{\left( 1-t\right)
^{\alpha }}-1\right) .  \label{Term2}
\end{equation}%
On the other hand, resolving (\ref{Def_t}) with respect to $\lambda ^{\alpha
}$ we get%
\begin{equation}
\lambda ^{-\alpha }=\frac{1}{\left( 1-t\right) ^{\alpha }}-1.  \label{Term3}
\end{equation}

Combination of (\ref{Def_t}), (\ref{Term2}) and (\ref{Term3}) proves (\ref%
{Term1}) and, therefore, (\ref{Term0}).

Lemma \ref{L_proper} is proved.

We note that under conditions (\ref{FinVariance}) a statement similar to
Theorem \ref{T_reduced} was proved in \cite{LV2018}.

For further references we would like to mention that if $mn^{-1}\rightarrow 0$ as $n\rightarrow \infty $ then it follows
from (\ref{SurvivalProbab}) and properties of slowly varying functions that
\begin{equation}
\frac{1-f_{n}(0)}{1-f_{m}(0)}\rightarrow 0  \label{Slowly}
\end{equation}%
as $n\rightarrow \infty $. Indeed, since the function $L^*(n)$\thinspace\ from (\ref{SurvivalProbab}) is
slowly varying, there exists an $a>0$ and functions $\,\gamma(x)$ and
$\,\theta(x)$\thinspace\ satisfying $\,\gamma(x)\rightarrow\gamma\in\left(
0,\infty\right)  $\thinspace\ and $\,\theta(x)\rightarrow0$\thinspace\ as
$x\uparrow\infty,$\thinspace\ such that (see \cite[Section~1.5]{Sen76})
\[
L^*(x) =\ \gamma(n)\exp\!\Big[\int_{a}^{n}\frac{\theta
(x)}{x}\,\mathrm{d}x\,\Big ].
\]
Hence, it follows easily that for any $\epsilon>0,$\thinspace\ there exists
$\,w=w(\epsilon)$\thinspace\ such that
$$
\Big(\frac{n}{m}\Big)^{-\epsilon}\,\leq\, \frac{L^*(n)}{L^*(m)} \leq \Big(\frac{n}{m}\Big)^{\epsilon} \quad\text{as }\,m/n\leq w.
$$
Evidently, these estimates combined with (\ref{SurvivalProbab})  imply (\ref{Slowly}).
\section{Population size}

We now prove Theorem \ref{T_Population_size}. It follows from Corollary \ref%
{Cor1} and (\ref{LOcal}) that if $n-m\gg \varphi (n)\rightarrow \infty $
then
\begin{equation}
\mathbf{E}\left[ e^{-\lambda Z(m)/T}|\mathcal{H}(n,\varphi (n))\right] =%
\mathbf{E}\left[ e^{-\lambda Z(m)/T};Z(m,n)=1|\mathcal{H}(n,\varphi (n))%
\right] +o(1)  \label{MRCA_limit}
\end{equation}%
for any $T>0$ and any $\lambda $ with $Re\lambda >0$. Further we note that
if $Z(m,n)=1$ then all particles of the $n$-th generation are descendants of
the only particle of the $m$-th generation. Therefore,
\begin{eqnarray*}
&& \mathbf{E}\left[ e^{-\lambda Z(m)/T};Z(m,n)=1,\mathcal{H}(n,\varphi
(n))\right]\\
&&\qquad\qquad\qquad\quad=\mathbf{E}\left[ Z(m)e^{-\lambda Z(m)/T}f_{n-m}^{Z(m)-1}(0)\right]
\mathbf{P}\left( \mathcal{H}(n-m,\varphi (n))\right) .
\end{eqnarray*}
Hence, using Theorem \ref{T_deviation} we conclude that if condition (\ref%
{Def_phi}) holds and $n-m\gg \varphi (n)\rightarrow \infty $ then\
\begin{eqnarray}
&&\mathbf{E}\left[ e^{-\lambda Z(m)/T}|\mathcal{H}(n,\varphi (n))\right]
\notag \\
&&\qquad=\mathbf{E}\left[ Z(m)e^{-\lambda Z(m)/T}f_{n-m}^{Z(m)-1}(0)\right]
\frac{\mathbf{P}\left( \mathcal{H}(n-m,\varphi (n))\right) }{\mathbf{P}%
\left( \mathcal{H}(n,\varphi (n))\right) }+o(1)  \notag \\
&&\qquad=\mathbf{E}\left[ Z(m)e^{-\lambda Z(m)/T}f_{n-m}^{Z(m)-1}(0)\right]
\frac{1-f_{n-m}(0)}{1-f_{n}(0)}\frac{n}{n-m}+o(1).  \label{Laplace_size}
\end{eqnarray}%
We will use (\ref{Laplace_size}) to prove points 1-3 of Theorem \ref%
{T_Population_size}.

\subsection{The distribution of $Z(m)$ given $m=o(n)$ and $n\gg \protect%
\varphi (n)$}

We set $T^{-1}:=1-f_{m}(0)$ and, observing that
\begin{equation*}
\frac{1-f_{n-m}(0)}{1-f_{n}(0)}\frac{n}{n-m}\rightarrow 1
\end{equation*}%
if $m=o(n)$ and $n\rightarrow \infty ,$ rewrite (\ref{Laplace_size}) as
\begin{equation*}
\mathbf{E}\left[ e^{-\lambda Z(m)/T}|\mathcal{H}(n,\varphi (n))\right]
=(1+o(1))\mathbf{E}\left[ Z(m)e^{-\lambda Z(m)/T}f_{n-m}^{Z(m)-1}(0)\right]
+o(1).
\end{equation*}%
Further,
\begin{equation*}
\mathbf{E}\left[ Z(m)e^{-\lambda Z(m)/T}f_{n-m}^{Z(m)-1}(0)\right] =\frac{1}{%
f_{n-m}(0)}\mathbf{E}\left[ Z(m)e^{-(\lambda -T\log f_{n-m}(0))Z(m)/T}\right]
.
\end{equation*}%
It view of the equivalence $-\log f_{n-m}(0)\sim 1-f_{n-m}(0)$ as $%
n-m\rightarrow \infty $ and the relation%
\begin{equation*}
T\left( 1-f_{n-m}(0)\right) =\frac{1-f_{n-m}(0)}{1-f_{m}(0)}\rightarrow 0
\end{equation*}%
valid according to (\ref{SurvivalProbab}) if $m=o(n),$ we conclude that, for
any $\varepsilon >0$
\begin{eqnarray}
\frac{1}{f_{n-m}(0)}\mathbf{E}\left[ Z(m)e^{-\left( \lambda +\varepsilon
\right) Z(m)/T}\right] &\leq &\mathbf{E}\left[ Z(m)e^{-\lambda
Z(m)/T}f_{n-m}^{Z(m)-1}(0)\right]  \notag \\
&\leq &\mathbf{E}\left[ Z(m)e^{-\lambda Z(m)/T}\right]
\label{AboveBelow1}
\end{eqnarray}%
for sufficiently large $n$ and $m$. We know by (\ref{Yaglom000}) that%
\begin{equation}
\lim_{m\rightarrow \infty }\mathbf{E}\left[ e^{-\lambda
(1-f_{m}(0))Z(m)}|Z(m)>0\right] =1-\frac{1}{\left( 1+\lambda ^{-\alpha
}\right) ^{1/\alpha }}  \label{Yagl}
\end{equation}%
for all $\lambda >0$. Since the prelimiting functions at the left-hand side
of (\ref{Yagl}) are analytic functions in the domain $Re\lambda >0$ and the
function%
\begin{equation*}
G(\lambda ):=1-\frac{1}{\left( 1+\lambda ^{-\alpha }\right) ^{1/\alpha }}
\end{equation*}%
admits an analytic continuation in the same domain if we take the branch
with $G(1)=1-2^{-1/\alpha }$, it follows from the uniqueness theorem for analytic
functions that (\ref{Yagl}) holds for this branch for all $\lambda $ satisfying $Re\lambda >0
$. This, in turn, implies that the first derivatives of the preliminary
functions converge to the first derivative of the  selected branch $G(\lambda)$ in the domain $Re\lambda >0.$  In particular,
\begin{eqnarray}
&&\lim_{m\rightarrow \infty }\frac{\partial }{\partial \lambda }\mathbf{E}%
\left[ e^{-\lambda \left( 1-f_{m}(0)\right) Z(m)}|Z(m)>0\right]  \notag \\
&&\qquad \qquad =\lim_{m\rightarrow \infty }\frac{\partial }{\partial
\lambda }\frac{\mathbf{E}\left[ e^{-\lambda \left( 1-f_{m}(0)\right) Z(m)}%
\right] -f_{m}(0)}{1-f_{m}(0)}  \notag \\
&&\qquad \qquad =\lim_{m\rightarrow \infty }\frac{1}{1-f_{m}(0)}\frac{%
\partial }{\partial \lambda }\mathbf{E}\left[ e^{-\lambda \left(
1-f_{m}(0)\right) Z(m)}\right]  \notag \\
&&\qquad \qquad =-\lim_{m\rightarrow \infty }\mathbf{E}\left[
Z(m)e^{-\lambda \left( 1-f_{m}(0)\right) Z(m)}\right]  \notag \\
&&\qquad \qquad =\frac{\partial }{\partial \lambda }\left( 1-\frac{1}{\left(
1+\lambda ^{-\alpha }\right) ^{1/\alpha }}\right) =-\frac{1}{\left(
1+\lambda ^{a}\right) ^{1/\alpha +1}}  \label{Asym1}
\end{eqnarray}%
for any $\lambda >0$. Hence, using (\ref{AboveBelow1}) and letting $%
\varepsilon $ tend to zero we conclude that%
\begin{eqnarray*}
\lim_{n\gg m\rightarrow \infty }\mathbf{E}\left[ e^{-\lambda Z(m)/T}|%
\mathcal{H}(n,\varphi (n))\right] &=&\lim_{n\gg m\rightarrow \infty }\mathbf{%
E}\left[ Z(m)e^{-\lambda Z(m)/T}f_{n-m}^{Z(m)-1}(0)\right] \\
&=&\frac{1}{\left( 1+\lambda ^{a}\right) ^{1/\alpha +1}}.
\end{eqnarray*}

This proves point 1) of Theorem \ref{T_Population_size}.

\subsection{The distribution of $Z(m)$ if $m\sim \protect\theta n$ for some $%
\protect\theta \in (0,1)$\newline
}

In this section and in the sections to follow we consider the products of
the type $\theta n$ and $(1-\theta )n$ as $\left[ \theta n\right] $ and $%
\left[ (1-\theta )n\right] $. Setting, as before, $T^{-1}:=1-f_{m}(0)$ we
deduce by (\ref{SurvivalProbab}) that if $m\sim \theta n\rightarrow \infty $
then
\begin{equation*}
T\log f_{n-m}(0)\sim -\frac{1-f_{n-m}(0)}{1-f_{m}(0)}\sim -\left( \frac{%
\theta }{1-\theta }\right) ^{1/\alpha }.
\end{equation*}

Observe now that relations (\ref{MRCA_limit}) and (\ref{Laplace_size}) are still valid if $m\sim \theta n$ for some $%
\theta \in (0,1)$. Using these relations we conclude that to prove point 2)
of Theorem \ref{T_Population_size} it is sufficient to find, for each $%
\lambda >0$ the limit of the product
\begin{eqnarray}
&&\mathbf{E}\left[ Z(m)e^{-(\lambda -T\log f_{n-m}(0))Z(m)/T}\right] \frac{%
\mathbf{P}\left( \mathcal{H}(n-m,\varphi (n))\right) }{\mathbf{P}\left(
\mathcal{H}(n,\varphi (n))\right) }  \notag \\
&&\quad \sim \mathbf{E}\left[ Z(m)e^{-\left( \lambda +\left( \frac{\theta }{%
1-\theta }\right) ^{1/\alpha }\right) Z(m)(1-f_{m}(0))}\right] \frac{\mathbf{%
P}\left( \mathcal{H}(n-m,\varphi (n))\right) }{\mathbf{P}\left( \mathcal{H}%
(n,\varphi (n))\right) }  \label{Prel2}
\end{eqnarray}%
as $m\sim \theta n\rightarrow \infty $. Since
\begin{equation*}
\frac{\mathbf{P}\left( \mathcal{H}(n-m,\varphi (n))\right) }{\mathbf{P}%
\left( \mathcal{H}(n,\varphi (n))\right) }\sim \frac{1-f_{n-m}(0)}{n-m}\frac{%
n}{1-f_{n}(0)}\sim \left( \frac{1}{1-\theta }\right) ^{1/\alpha +1}
\end{equation*}%
in view of Theorem \ref{T_deviation} and (\ref{SurvivalProbab}), we deduce
by the last two lines in (\ref{Asym1}) and relation~(\ref{Prel2}) that%
\begin{eqnarray*}
&&\lim_{m\sim \theta n\rightarrow \infty }\mathbf{E}\left[ e^{-\lambda
(1-f_{m}(0))Z(m)}|\mathcal{H}(n,\varphi (n))\right] \\
&&\qquad=\lim_{m\sim \theta n\rightarrow \infty }\mathbf{E}\left[
Z(m)e^{-(\lambda -T\log f_{n-m}(0))Z(m)/T}\right] \frac{\mathbf{P}\left(
\mathcal{H}(n-m,\varphi (n))\right) }{\mathbf{P}\left( \mathcal{H}(n,\varphi
(n))\right) } \\
&&\qquad \qquad =\frac{1}{\left( 1+\left( \lambda +\left( \frac{\theta }{%
1-\theta }\right) ^{1/\alpha }\right) ^{\alpha }\right) ^{1/\alpha +1}}%
\left( \frac{1}{1-\theta }\right) ^{1/\alpha +1} \\
&&\qquad \qquad =\frac{1}{\left( 1-\theta +\left( \lambda \left( 1-\theta
\right) ^{1/\alpha }+\theta ^{1/\alpha }\right) ^{\alpha }\right) ^{1/\alpha
+1}}
\end{eqnarray*}

This proves point 2) of Theorem \ref{T_Population_size}.

\subsection{The distribution of $Z(m)$ given $m=n-\protect\psi (n)$ and $%
n\gg \protect\psi (n)\gg \protect\varphi (n)$}

We start this section by proving the following statement.

\begin{lemma}
\label{L_represent}Let condition (\ref{MainAssump}) be valid and $\rho >0$
be a fixed number. If $l\rightarrow \infty $ then the parameter $r$
satisfying the inequalities%
\begin{equation}
1-f_{r+1}(0)\leq 1-f_{l}^{\rho }(0)<1-f_{r}(0)  \label{R_inequality}
\end{equation}%
has the asymptotic representation%
\begin{equation}
r\sim l\rho ^{-\alpha },\quad l\rightarrow \infty .  \label{Assym_r}
\end{equation}
\end{lemma}

\textbf{Proof}. In view of (\ref{SurvivalProbab}) and properties of slowly
varying functions
$
1-f_{l}^{\rho }(0)\sim \rho \left( 1-f_{l}(0)\right) \sim 1-f_{l\rho
^{-\alpha }}(0)
$
as $l\rightarrow \infty $. Hence it follows that if $r$ satisfies (\ref%
{R_inequality}) then $r\rightarrow \infty $. Besides,
\begin{equation*}
\lim_{r\rightarrow \infty }\frac{1-f_{r+1}(0)}{1-f_{r}(0)}%
=\lim_{r\rightarrow \infty }\frac{1-f(f_{r}(0))}{1-f_{r}(0)}=\mathbf{E}\xi
=1.
\end{equation*}%
Therefore,
\begin{equation*}
1-f_{l\rho ^{-\alpha }}(0)\sim 1-f_{r}(0)
\end{equation*}%
as $l\rightarrow \infty $. Using now (\ref{SurvivalProbab}) once again we
get (\ref{Assym_r}).

Lemma \ref{L_represent} is proved.

Our aim is to show that $m=n-\psi (n)$ and $n\gg \psi (n)\gg \varphi
(n)\rightarrow \infty $ then, for any $\lambda >0$%
\begin{equation}
\lim_{n\gg \psi (n)\gg \varphi (n)\rightarrow \infty }\mathbf{E}\left[
e^{-\lambda (1-f_{n-m}(0))Z(m)}|\mathcal{H}(n,\varphi (n))\right] =\frac{1}{%
\left( 1+\lambda \right) ^{\alpha +1}}.  \label{Intermed}
\end{equation}

Note that by Theorem \ref{T_deviation}%
\begin{equation*}
\mathbf{P}\left( \mathcal{H}(l,r)\right) \sim \frac{1-f_{l}(0)}{\alpha l}%
\frac{r}{\Gamma \left( 1+\alpha \right) }
\end{equation*}%
if $l\gg r\rightarrow \infty $. Besides, given $n\gg \varphi (n)\rightarrow
\infty $ the representation (\ref{Laplace_size}) is valid. We now set
\begin{equation}
T=-\frac{1}{\log f_{n-m}(0)}\sim \frac{1}{1-f_{n-m}(0)}  \label{Assym_T}
\end{equation}%
and write, for any complex $\lambda $ with $Re\lambda >0$ the equality
\begin{equation*}
\mathbf{E}\left[ Z(m)e^{-\lambda Z(m)/T}f_{n-m}^{Z(m)-1}(0)\right] =\frac{1}{%
f_{n-m}(0)}\mathbf{E}\left[ Z(m)\left( f_{n-m}^{1+\lambda }(0)\right) ^{Z(m)}%
\right] .
\end{equation*}

To investigate the asymptotic behavior of the expectation at the right-hand
side of the equality we first consider the function%
\begin{equation*}
\mathbf{E}\left[ e^{\left( 1+\lambda \right) Z(m)\log f_{n-m}(0)}\right]
=f_{m}\left( f_{n-m}^{1+\lambda }(0)\right)
\end{equation*}%
and find, for $\lambda >0$ the parameter $r=r(\lambda )$ such that%
\begin{equation}
1-f_{r+1}(0)\leq 1-f_{n-m}^{1+\lambda }(0)<1-f_{r}(0).  \label{DEf_r}
\end{equation}%
According to Lemma \ref{L_represent}%
\begin{equation}
r\sim \frac{n-m}{\left( 1+\lambda \right) ^{\alpha }}  \label{Asympt_r}
\end{equation}%
as $n-m\rightarrow \infty $. Using (\ref{DEf_r}) and properties of
iterations we conclude that, for $\lambda >0$%
\begin{equation*}
f_{m+r}(0)-f_{m}(0)\leq \mathbf{E}\left[ e^{\left( 1+\lambda \right)
Z(m)\log f_{n-m}(0)}\right] -f_{m}(0)\leq f_{m+r+1}(0)-f_{m}(0).
\end{equation*}%
Set
\begin{equation*}
\Delta \left( k\right) =\frac{1-f_{k}(0)}{\alpha k},\;k=1,2,....
\end{equation*}%
In view of (\ref{SurvivalProbab}) and properties of slowly varying functions
\begin{equation*}
\frac{\Delta (m+k+1)}{\Delta (n)}\rightarrow 1
\end{equation*}%
as $n,m\rightarrow \infty $ uniformly in $k\leq n-m=\psi (n)=o(n)$. Besides,
\begin{eqnarray*}
\lim_{k\rightarrow \infty }\frac{f_{k+1}(0)-f_{k}(0)}{\Delta \left(
k+1\right) } &=&\lim_{k\rightarrow \infty }\frac{f\left( f_{k}(0)\right)
-f_{k}(0)}{\Delta \left( k+1\right) } \\
&=&\lim_{k\rightarrow \infty }\frac{\alpha k\left( 1-f_{k}(0)\right)
^{1+\alpha }L(1-f_{k}(0))}{1-f_{k}(0)} \\
&=&\lim_{k\rightarrow \infty }\alpha k\left( 1-f_{k}(0)\right) ^{\alpha
}L(1-f_{k}(0))=1
\end{eqnarray*}%
in view of (\ref{MainAssump}) and (\ref{SimpeForm}). Hence, using (\ref%
{Asympt_r}) we get
\begin{eqnarray*}
&&\lim_{n\rightarrow \infty }\frac{1}{n-m}\left[ \frac{%
f_{m}(f_{r}(0))-f_{m}(0)}{\Delta \left( n\right) }\right] \\
&&\qquad \qquad =\lim_{n\rightarrow \infty }\frac{1}{n-m}\sum_{k=0}^{r-1}%
\frac{f_{m}(f_{k+1}(0))-f_{m}(f_{k}(0))}{\Delta \left( n\right) } \\
&&\qquad \qquad =\lim_{n\rightarrow \infty }\frac{1}{n-m}\sum_{k=0}^{r-1}%
\frac{\Delta \left( m+k+1\right) }{\Delta \left( n\right) }\left[ \frac{%
f_{m+k+1}(0)-f_{m+k}(0)}{\Delta \left( m+k+1\right) }\right] \\
&&\qquad \qquad =\lim_{n\rightarrow \infty }\frac{1}{n-m}\sum_{k=0}^{r-1}1=%
\frac{1}{(1+\lambda )^{\alpha }}.
\end{eqnarray*}

Thus,
\begin{equation}
\lim_{n\rightarrow \infty }\frac{1}{n-m}\frac{\mathbf{E}\left[ e^{\left(
1+\lambda \right) Z(m)\log f_{n-m}(0)}\right] -f_{m}(0)}{\Delta \left(
n\right) }=\frac{1}{(1+\lambda )^{\alpha }}  \label{Yagl2}
\end{equation}%
for any $\lambda >0$. Now we apply the same arguments as before. Since the
prelimiting functions at the left-hand side of (\ref{Yagl2}) are analytic
functions in the complex domain $Re\lambda >0$ and the function%
\begin{equation*}
G(\lambda ):=\frac{1}{(1+\lambda )^{\alpha }},\ \lambda >0,
\end{equation*}%
admits an analytic continuation to the same domain if we take the branch
with $G(1)=\left( 1/2\right) ^{\alpha }>0$, it follows from the uniqueness
theorem for analytic functions that (\ref{Yagl2}) holds for this branch in the domain $%
Re\lambda >0$. This, in turn, implies that the first derivatives of the
preliminary functions with respect to $\lambda $ converge, as $n\rightarrow
\infty $ to the first derivative of the selected branch for $Re\lambda >0.$ Therefore,%
\begin{eqnarray*}
&&\lim_{n\rightarrow \infty }\frac{\partial }{\partial \lambda }\frac{1}{n-m}%
\frac{\mathbf{E}\left[ e^{\left( 1+\lambda \right) Z(m)\log f_{n-m}(0)}%
\right] -f_{m}(0)}{\Delta \left( n\right) } \\
&&\qquad \qquad =\lim_{n\rightarrow \infty }\frac{\partial }{\partial
\lambda }\frac{1}{n-m}\left[ \frac{f_{m}(f_{r}(0))-f_{m}(0)}{\Delta \left(
n\right) }\right] \\
&&\qquad \qquad =\lim_{n\rightarrow \infty }\frac{1}{(n-m)\Delta \left(
n\right) }\frac{\partial }{\partial \lambda }f_{m}\left( e^{\left( 1+\lambda
\right) \log f_{n-m}(0)}\right) \\
&&\qquad \qquad =-\lim_{n\rightarrow \infty }\frac{1}{(n-m)\Delta \left(
n\right) T}\mathbf{E}\left[ Z(m)e^{-\left( 1+\lambda \right) Z(m)/T}\right]
=-\frac{\alpha }{\left( 1+\lambda \right) ^{\alpha +1}}.
\end{eqnarray*}%
Thus,%
\begin{equation*}
\mathbf{E}\left[ Z(m)e^{-\left( 1+\lambda \right) Z(m)/T}\right] \sim \frac{%
\alpha }{\left( 1+\lambda \right) ^{\alpha +1}}\frac{(n-m)\Delta \left(
n\right) T}{1}.
\end{equation*}

\bigskip \bigskip Now%
\begin{eqnarray*}
&&\mathbf{E}\left[ Z(m)e^{-\left( 1+\lambda \right) Z(m)/T}\right] \frac{%
1-f_{n-m}(0)}{n-m}\frac{n}{1-f_{n}(0)} \\
&&\qquad \qquad \sim \frac{\alpha }{\left( 1+\lambda \right) ^{\alpha +1}}%
\frac{1-f_{n-m}(0)}{n-m}\frac{n}{1-f_{n}(0)}\frac{T(n-m)\Delta \left(
n\right) }{1} \\
&&\qquad\qquad \qquad =\frac{1}{\left( 1+\lambda \right) ^{\alpha +1}}.
\end{eqnarray*}%
Hence, using (\ref{Laplace_size}) with $T^{-1}=-\log f_{n-m}(0)$ and (\ref%
{Assym_T}) we conclude that%
\begin{equation*}
\lim_{n\rightarrow \infty }\mathbf{E}\left[ e^{-\lambda (1-f_{n-m}(0))Z(m)}|%
\mathcal{H}(n,\varphi (n))\right] =\frac{1}{\left( 1+\lambda \right)
^{\alpha +1}},\;\lambda >0,
\end{equation*}%
as needed.

\subsection{The distribution of $Z(m)$ for $n-m\sim y\protect\varphi (n)$, $%
y>0$}

In this section we consider the case when the gap between $n$ and $m$ is of
order $\varphi (n)$ and consider $y\varphi (n)$ as $\left[ y\varphi (n)%
\right] $. We denote by $Z_{1}^{\ast }(l),\ldots ,Z_{j}^{\ast }(l),$ $j,l\in
\mathbb{N},$ independent random variables, each of which is distributed as $%
\left\{ Z(l)|Z(l)>0\right\} ,$ set
\begin{equation*}
\hat{Z}_{j}^{\ast }(l)=\sum_{r=1}^{j}Z_{r}^{\ast }(l),
\end{equation*}%
and introduce independent random variables $\eta _{1},\ldots ,\eta _{j}$
each of which has distribution $M(x)$ specified by (\ref{Yaglom000}).

For $Re \lambda >0$ and $T>0$ we have
\begin{eqnarray*}
&&\mathbf{E}\left[ e^{-\lambda Z(m)/T};\mathcal{H}(n,\varphi (n))\right] \\
&&\quad=\sum_{j=1}^{\infty }\mathbf{E}\left[ e^{-\lambda Z(m)/T};Z(m,n)=j,%
\mathcal{H}(n,\varphi (n))\right] \\
&&\quad=\sum_{j=1}^{\infty }\mathbf{E}\left[ e^{-\lambda Z(m)/T};Z(m,n)=j%
\right] \mathbf{P}\left( \mathcal{H}(n,\varphi (n))|Z(m,n)=j\right) \\
&&\quad=\sum_{j=1}^{\infty }\mathbf{E}\left[ e^{-\lambda Z(m)/T};Z(m,n)=j%
\right] \mathbf{P}\left( \left( 1-f_{\varphi (n)}(0)\right) \hat{Z}%
_{j}^{\ast }(y\varphi (n))<1\right) .
\end{eqnarray*}

It is not difficult to check, using (\ref{Yaglom000}) and (\ref%
{SurvivalProbab}) that, for any $j\geq 1$
\begin{eqnarray*}
&&\lim_{n\rightarrow \infty }\mathbf{P}\left( \mathcal{H}(n,\varphi
(n))|Z(m,n)=j\right) =\lim_{n\rightarrow \infty }\mathbf{P}\left(
(1-f_{\varphi (n)}(0))\hat{Z}_{j}^{\ast }(y\varphi (n))\leq 1\right) \\
&&\qquad\qquad =\lim_{n\rightarrow \infty }\mathbf{P}\left(
\sum_{r=1}^{j}\left( 1-f_{y\varphi (n)}(0)\right) Z_{r}^{\ast }(y\varphi
(n))\leq \frac{1-f_{y\varphi (n)}(0)}{1-f_{\varphi (n)}(0)}\right) \\
&&\qquad\qquad\qquad =\mathbf{P}\left( \eta _{1}+\cdots +\eta _{j}\leq
y^{-1/\alpha }\right) =M^{\ast j}(y^{-1/\alpha }).
\end{eqnarray*}

Taking now $T^{-1}=-\log f_{y\varphi (n)}(0)$ we evaluate, for each $j$ the
quantity
\begin{eqnarray*}
&&\mathbf{E}\left[ e^{-\lambda Z(m)/T};Z(m,n)=j\right] \\
&&\qquad =\sum_{k=j}^{\infty }e^{-\lambda k/T}\mathbf{P}\left( Z(m)=k\right)
\binom{k}{j}f_{y\varphi (n)}^{k-j}(0)\left( 1-f_{y\varphi (n)}(0)\right) ^{j}
\\
&&\qquad =\frac{\left( 1-f_{y\varphi (n)}(0)\right) ^{j}}{j!}e^{-\lambda
j/T}\sum_{k=j}^{\infty }e^{-\lambda \left( k-j\right) /T}\mathbf{P}\left(
Z(m)=k\right) k^{[j]}f_{y\varphi (n)}^{k-j}(0) \\
&&\qquad =\frac{\left( 1-f_{y\varphi (n)}(0)\right) ^{j}}{j!}e^{-\lambda
j/T}\sum_{k=j}^{\infty }\mathbf{P}\left( Z(m)=k\right) k^{[j]}\left(
f_{y\varphi (n)}^{1+\lambda }(0)\right) ^{k-j}.
\end{eqnarray*}%
Clearly,%
\begin{equation*}
\sum_{k=j}^{\infty }k^{\left[ j\right] }\mathbf{P}\left( Z(m)=k\right)
\left( f_{y\varphi (n)}^{1+\lambda }(0)\right) ^{k-j}=\frac{d^{j}f_{m}(s)}{%
ds^{j}}\left\vert _{s=f_{y\varphi (n)}^{1+\lambda }(0)}\right.
=f_{m}^{(j)}\left( f_{y\varphi (n)}^{1+\lambda }(0)\right) .
\end{equation*}%
According to (\ref{Yagl2}) if\textbf{\ }$n-m=y\varphi (n),y>0,$ and $\varphi
(n)=o(n)$ as $n\rightarrow \infty $ then\textbf{\ }%
\begin{equation}
\lim_{n\rightarrow \infty }\frac{f_{m}(e^{\left( 1+\lambda \right) \log
f_{y\varphi (n)}(0)})-f_{m}(0)}{y\varphi (n)\Delta \left( n\right) }=\frac{1%
}{\left( 1+\lambda \right) ^{\alpha }},\quad \lambda >0.  \label{Complex1}
\end{equation}%
Since the prelimiting functions at the left-hand side of (\ref{Complex1})
have an analytic continuation in the domain $Re\lambda >-1$ and the same is
true for the function $\left( 1+\lambda \right) ^{-\alpha }$ with the
agreement that $1^{-\alpha }=1$, the derivatives of any order with respect
to $\lambda $ of the prelimiting functions converge, as $n\rightarrow \infty
$ to the derivative of the respective order of the limiting function for any
$Re\lambda >-1$. Hence it follows that, for all $\lambda >0$ and each $J\geq
1$
\begin{equation}
\lim_{n\rightarrow \infty }\frac{1}{y\varphi (n)\Delta \left( n\right) }%
\frac{\partial ^{J}}{\partial ^{J}\lambda }\left[ f_{m}(e^{\left( 1+\lambda
\right) \log f_{y\varphi (n)}(0)})\right] =\left( -1\right) ^{J}\frac{\Gamma
(\alpha +J)}{\Gamma (\alpha )}\frac{1}{\left( 1+\lambda \right) ^{\alpha +J}}%
.  \label{Derivative2}
\end{equation}

In particular,
\begin{eqnarray}
&&\frac{1}{y\varphi (n)\Delta \left( n\right) }\frac{d}{d\lambda }\left[
f_{m}(e^{\left( 1+\lambda \right) \log f_{y\varphi (n)}(0)})\right]  \notag
\\
&&\qquad\qquad\qquad\qquad=\frac{\log f_{y\varphi (n)}(0)}{y\varphi
(n)\Delta \left( n\right) }f_{m}^{\prime }(f_{y\varphi (n)}^{1+\lambda
}(0))f_{y\varphi (n)}^{1+\lambda }(0)  \notag \\
&&\qquad\qquad\qquad\qquad \sim \frac{\log f_{y\varphi (n)}(0)}{y\varphi
(n)\Delta \left( n\right) }f_{m}^{\prime }(f_{y\varphi (n)}^{1+\lambda }(0))
\notag \\
&&\qquad\qquad\qquad\qquad \sim \frac{\left( -1\right) \alpha }{\left(
1+\lambda \right) ^{\alpha +1}}=\frac{\left( -1\right) }{\left( 1+\lambda
\right) ^{\alpha +1}}\frac{\Gamma (\alpha +1)}{\Gamma (\alpha )}.
\label{FirstStep}
\end{eqnarray}%
as $n\rightarrow \infty $.

Now we show that, for any $k\geq 1$
\begin{equation}
\frac{\left( \log f_{y\varphi (n)}(0)\right) ^{k}}{y\varphi (n)\Delta \left(
n\right) }f_{m}^{(k)}\left( f_{x\varphi (n)}^{1+\lambda }(0)\right) \sim
\frac{\left( -1\right) ^{k}}{\left( 1+\lambda \right) ^{\alpha +k}}\frac{%
\Gamma (\alpha +k)}{\Gamma (\alpha )}  \label{BasicInduction}
\end{equation}%
as $n\rightarrow \infty $. To this aim we use the Fa\`{a} di Bruno formula
for the derivatives of the composition of functions $F(\cdot )$ and $G(\cdot
)$ in the following form:
\begin{equation*}
\frac{d^{J}}{dz^{J}}\left[ F(G(z))\right] =%
\sum_{k=1}^{J}F^{(k)}(G(z))B_{J,k}\left( G^{\prime }(z),G^{\prime \prime
}(z),\ldots ,G^{(J-k+1)}(z)\right) ,
\end{equation*}%
where, for $1\leq k\leq J$
\begin{equation*}
B_{J,k}(x_{1},\ldots ,x_{J-k+1}):=\sum_{\Pi \left( J,k\right) }\frac{J!}{%
i_{1}!\cdots i_{J-k+1}!}\prod_{r=1}^{J-k+1}\left( \frac{x_{r}}{r!}\right)
^{i_{r}}  \label{DefB_Jk}
\end{equation*}%
are the so-called Bell polynomials of the second kind. These polynomials
possess the following properties:

For any $a$ and $b$%
\begin{equation*}
B_{J,k}(ab,ab^{2},\ldots ,ab^{J-k+1})=a^{k}b^{J}B_{J,k}(1,1,\ldots ,1)
\end{equation*}
and
\begin{equation}
B_{J,k}(1,1,\ldots ,1)=S(J,k)=\frac{1}{k!}\sum_{r=1}^{k}\left( -1\right)
^{k-r}\binom{k}{r}r^{J},  \label{StirlingNumbers}
\end{equation}%
where $S(J,k),\ k=1,2,\ldots ,J$ denote the Stirling numbers of the second
kind. 

Setting $F(z):=f_{m}(z)$ and $G(\lambda ):=f_{y\varphi (n)}^{1+\lambda }(0)$
and observing that%
\begin{equation*}
\frac{d^{r}G(\lambda )}{d^{r}\lambda }=f_{y\varphi (n)}^{1+\lambda }(0)\log
^{r}f_{y\varphi (n)}(0),r=1,2,...,
\end{equation*}%
and, therefore,%
\begin{eqnarray*}
&&B_{J,k}\left( G^{\prime }(z),G^{\prime \prime }(z),\ldots
,G^{(J-k+1)}(z)\right) \\
&=&\left( f_{y\varphi (n)}^{1+\lambda }(0)\right) ^{k}B_{J,k}\left( \log
f_{y\varphi (n)}(0),\log ^{2}f_{y\varphi (n)}(0),\ldots ,\left( \log
f_{y\varphi (n)}(0)\right) ^{J-k+1}\right) \\
&=&\left( f_{y\varphi (n)}^{1+\lambda }(0)\right) ^{k}\left( \log
f_{y\varphi (n)}(0)\right) ^{J}B_{J,k}\left( 1,1,\ldots ,1\right) ,
\end{eqnarray*}%
we conclude that, for any fixed $y>0$
\begin{equation*}
\frac{d^{J}}{d^{J}\lambda }\left[ f_{m}(e^{(1+\lambda )\log f_{y\varphi
(n)}(0)})\right] =\left( \log f_{y\varphi (n)}(0)\right)
^{J}\sum_{k=1}^{J}f_{m}^{(k)}(f_{y\varphi (n)}^{1+\lambda }(0))\left(
f_{y\varphi (n)}^{1+\lambda }(0)\right) ^{k}S(J,k).
\end{equation*}

Now we prove (\ref{BasicInduction}) by induction. By (\ref{FirstStep}) this
equivalence is true for $k=1,$ \ and if it is true for all $k<J$ then, using
(\ref{Derivative2}) and convergence of $f_{y\varphi (n)}(0)$ to 1 and $\log
f_{y\varphi (n)}(0)$ to 0 as $n\rightarrow \infty,$ we conclude that, for
any $y>0$%
\begin{eqnarray*}
&&\frac{1}{y\varphi (n)\Delta \left( n\right) }\frac{d^{J}}{d^{J}\lambda }%
\left[ f_{m}(e^{(1+\lambda )\log f_{y\varphi (n)}(0)})\right] =\frac{\left(
\log f_{y\varphi (n)}(0)\right) ^{J}f_{m}^{(J)}(f_{y\varphi (n)}^{1+\lambda
}(0))}{y\varphi (n)\Delta \left( n\right) } \\
&&\qquad \qquad \qquad \qquad \qquad \quad \quad +\sum_{k=1}^{J-1}\frac{%
\left( \log f_{y\varphi (n)}(0)\right) ^{J}f_{m}^{(k)}(f_{y\varphi
(n)}^{1+\lambda }(0))}{y\varphi (n)\Delta \left( n\right) }S(J,k) \\
&&\qquad \qquad \qquad \qquad =\frac{\left( \log f_{y\varphi (n)}(0)\right)
^{J}f_{m}^{(J)}(f_{y\varphi (n)}^{1+\lambda }(0))}{y\varphi (n)\Delta \left(
n\right) }+O\left( \log f_{y\varphi (n)}(0)\right) \\
&&\qquad \qquad \qquad \qquad \sim \frac{\left( -1\right) ^{J}}{\left(
1+\lambda \right) ^{\alpha +J}}\frac{\Gamma (\alpha +J)}{\Gamma (\alpha )}
\end{eqnarray*}%
as $n\rightarrow \infty $. These estimates conclude the step of induction
and prove (\ref{BasicInduction}).

Now we have, for any fixed $y>0$ and $j\in\mathbb{N}$%
\begin{eqnarray*}
&&\mathbf{E}\left[ e^{-\lambda Z(m)/T};Z(m,n)=j\right] \\
&&\qquad \qquad =\frac{\left( 1-f_{y\varphi (n)}(0)\right) ^{j}}{j!}\frac{1}{%
f_{y\varphi (n)}^{\lambda j}(0)^{j}\left( \log f_{y\varphi (n)}(0)\right)
^{j}}\frac{d^{j}}{d^{j}\lambda }f_{m}\left( f_{y\varphi (n)}^{1+\lambda
}(0)\right) \\
&&\qquad \qquad \sim \frac{\left( 1-f_{y\varphi (n)}(0)\right) ^{j}}{%
j!f_{y\varphi (n)}^{\lambda j}(0)}\left( -1\right) ^{j}\frac{y\varphi
(n)\Delta (n)}{\left( \log f_{y\varphi (n)}(0)\right) ^{j}}\frac{\Gamma
\left( j+\alpha \right) }{\Gamma \left( \alpha \right) }\frac{1}{\left(
1+\lambda \right) ^{\alpha +j}} \\
&&\qquad \qquad \sim \frac{y\varphi (n)\Delta (n)}{j!}\frac{\Gamma \left(
j+\alpha \right) }{\Gamma \left( \alpha \right) }\frac{1}{\left( 1+\lambda
\right) ^{\alpha +j}}
\end{eqnarray*}%
as $n\rightarrow \infty $. Since
\begin{equation*}
\mathbf{P}\left( \mathcal{H}(n,\varphi (n))\right) \sim \frac{1-f_{n}(0)}{%
\alpha n}\frac{1}{\Gamma \left( 1+\alpha \right) }\varphi (n)=\Delta (n)%
\frac{\varphi (n)}{\Gamma \left( 1+\alpha \right) }
\end{equation*}%
as $n\rightarrow \infty $, we see that if $\varphi (n)=o(n),n-m=y\varphi
(n),y>0$, then for each fixed $j=1,2,...$
\begin{eqnarray*}
&&\lim_{\varphi (n)\rightarrow \infty }\mathbf{E}\left[ e^{-\lambda
Z(m)/T};Z(m,n)=j|\mathcal{H}(n,\varphi (n))\right] \\
&&\qquad=\lim_{\varphi (n)\rightarrow \infty }\mathbf{E}\left[ e^{-\lambda
Z(m)/T};Z(m,n)=j\right] \frac{\mathbf{P}\left( \mathcal{H}(n,\varphi
(n))|Z(m,n)=j\right) }{\mathbf{P}\left( \mathcal{H}(n,\varphi (n))\right) }
\\
&&\qquad \qquad =\frac{1}{\left( 1+\lambda \right) ^{\alpha +j}}\frac{\Gamma
\left( 1+\alpha \right) \Gamma \left( j+\alpha \right) }{j!\Gamma \left(
\alpha \right) }yM^{\ast j}(y^{-1/\alpha }) \\
&&\qquad \qquad =\frac{\alpha \Gamma \left( j+\alpha \right) }{j!\left(
1+\lambda \right) ^{\alpha +j}}yM^{\ast j}(y^{-1/\alpha })
\end{eqnarray*}%
Further, using Lemma \ref{L_proper} we get
\begin{eqnarray*}
&&\lim_{J\rightarrow \infty }\limsup_{\varphi (n)\rightarrow \infty }\mathbf{%
E}\left[ e^{-\lambda Z(m)/T};Z(m,n)>J|\mathcal{H}(n,\varphi (n))\right] \\
&&\qquad\leq\lim_{J\rightarrow \infty }\limsup_{\varphi (n)\rightarrow
\infty }\mathbf{P}\left( Z(m,n)>J|\mathcal{H}(n,\varphi (n))\right) \\
&&\qquad\qquad=1-\lim_{J\rightarrow \infty }\lim_{\varphi (n)\rightarrow
\infty }\mathbf{P}\left( Z(m,n)\leq J|\mathcal{H}(n,\varphi (n))\right) \\
&&\qquad\qquad=\lim_{J\rightarrow \infty }\sum_{j=J+1}^{\infty }\frac{\alpha
\Gamma \left( j+\alpha \right) }{j!}yM^{\ast j}(y^{-1/\alpha })=0.
\end{eqnarray*}%
Hence we conclude that
\begin{equation*}
\lim_{\varphi (n)\rightarrow \infty }\mathbf{E}\left[ e^{-\lambda Z(m)/T}|%
\mathcal{H}(n,\varphi (n))\right] =\sum_{j=1}^{\infty }\frac{\alpha \Gamma
\left( j+\alpha \right) }{j!\left( 1+\lambda \right) ^{\alpha +j}}yM^{\ast
j}(y^{-1/\alpha }),
\end{equation*}%
as needed.

\subsection{The distribution of $Z(m)$ if $m=n-\chi(n)$ where $n\gg
\varphi(n)\gg\chi (n)\geq 0$}

We need to show that if $m=n-\chi (n)$, where $n\gg \varphi (n)\gg \chi
(n)\geq 0$, then%
\begin{equation*}
\lim_{m\rightarrow \infty }\mathbf{E}\left[ e^{-\lambda \left( 1-f_{\varphi
(n)}(0)\right) Z(m)}|\mathcal{H}(n,\varphi (n))\right] =\alpha
\int_{0}^{1}x^{\alpha -1}e^{-\lambda x}dx.
\end{equation*}%
First observe that if
\begin{equation*}
T=1/(1-f_{\varphi (n)}(0))=o\left( 1/\left( 1-f_{n}(0)\right) \right)
\end{equation*}%
then we may use estimate (\ref{ModifDevoation}) and properties of slowly
varying functions to conclude that, for any $x\in (0,1]$
\begin{eqnarray}
&&\lim_{n\rightarrow \infty }\mathbf{P}\left( 0<Z(n)<xT\,|\,0<Z(n)<T\right)
=\lim_{n\rightarrow \infty }\frac{\mathbf{P}\left( 0<Z(n)<xT\right) }{%
\mathbf{P}\left( 0<Z(n)<T\right) }  \notag \\
&&\qquad \qquad \qquad \qquad \qquad \qquad \qquad \quad =\lim_{n\rightarrow
\infty }\frac{\left( xT\right) ^{\alpha }}{L\left( 1/xT\right) }\frac{%
L\left( 1/T\right) }{T^{\alpha }}=x^{\alpha }  \label{Distrib}
\end{eqnarray}%
and, therefore,%
\begin{equation}
\lim_{n\rightarrow \infty }\mathbf{E}\left[ e^{-\lambda Z(n)/T}|0<Z(n)<T%
\right] =\alpha \int_{0}^{1}x^{\alpha -1}e^{-\lambda x}dx.  \label{Laplace0}
\end{equation}

Now we consider the case $m=n-\chi (n)$ where $n\gg \varphi (n)\gg \chi
(n)\geq 1$. We fix $\varepsilon \in (0,1)$ and write the decomposition%
\begin{eqnarray*}
&&\mathbf{E}\left[ e^{-\lambda Z(m)/T};0<Z(n)<T\right] =\mathbf{E}\left[
e^{-\lambda Z(m)/T};0<Z(m)<\varepsilon T,0<Z(n)<T\right] \\
&&\qquad\qquad\qquad\qquad\qquad+\mathbf{E}\left[ e^{-\lambda Z(m)/T};Z(m)>NT,0<Z(n)<T\right] \\
&&\qquad\qquad\qquad\qquad\qquad+\mathbf{E}\left[ e^{-\lambda Z(m)/T};\varepsilon T\leq Z(m)\leq
NT,0<Z(n)<T\right].
\end{eqnarray*}
First we observe that
\begin{eqnarray}
&&\mathbf{E}\left[ e^{-\lambda Z(m)/T};0<Z(m)<\varepsilon T,0<Z(n)<T\right]
\notag \\
&&\qquad \qquad \qquad \leq \mathbf{P}(0<Z(m)<\varepsilon
T)=k_{m,n}(\varepsilon )\mathbf{P}(0<Z(n)<T),  \label{Remain1}
\end{eqnarray}%
where (see (\ref{SurvivalProbab}) and (\ref{ModifDevoation}))
\begin{equation*}
k_{m,n}(\varepsilon )=\frac{\mathbf{P}(0<Z(m)<\varepsilon T)}{\mathbf{P}%
(0<Z(n)<T)}\sim \frac{n}{m}\frac{1-f_{m}(0)}{1-f_{n}(0)}\frac{L(1/T)}{%
L(1/\left( \varepsilon T\right) )}\varepsilon ^{\alpha }\rightarrow
\varepsilon ^{\alpha }
\end{equation*}%
as $m\sim n$ $\rightarrow \infty $. Further, for any $N>1$%
\begin{equation}
\mathbf{E}\left[ e^{-\lambda Z(m)/T};Z(m)>NT,0<Z(n)<T\right] \leq e^{-N}%
\mathbf{P}(0<Z(n)<T).  \label{Remain2}
\end{equation}%
Since the right-hand sides in (\ref{Remain1}) and (\ref{Remain2}) are $%
o\left( \mathbf{P}(0<Z(n)<T)\right) $ as first $T\rightarrow \infty $ and
then $\varepsilon \downarrow 0$ and $N\uparrow \infty $, it remains to
evaluate the term%
\begin{eqnarray*}
&&\mathbf{E}\left[ e^{-\lambda Z(m)/T};\varepsilon T\leq Z(m)\leq NT,0<Z(n)<T%
\right]   \notag \\
&&\qquad \qquad \, =\sum_{j=\varepsilon T}^{NT}e^{-\lambda j/T}\mathbf{P}%
\left( Z(m)=j\right) \mathbf{P}(0<Z(\chi )<T|Z(0)=j)  \label{Remain3}
\end{eqnarray*}%
for sufficiently small $\varepsilon >0$ and sufficiently large $N$.

We note that if condition (\ref{MainAssump}) is valid then $\mathbf{E}%
\xi ^{r}=\mathbf{E}Z^{r}(1)<\infty $ for any $r\in (1,1+\alpha )$. This
estimate implies (see Theorem 6, Ch.2, $\S $ 3 in \cite{Sev74}) that $%
\mathbf{E}Z^{r}(n)<\infty $ for all $n\in \mathbb{N}$. Hence we conclude by
the Markov inequality that%
\begin{equation*}
\mathbf{P}(\left\vert Z(\chi )-j\right\vert >\delta j|Z(0)=j)\leq \frac{%
\mathbf{E}\left[ \left\vert Z(\chi )-j\right\vert ^{r}|Z(0)=j\right] }{%
\left( \delta j\right) ^{r}}<\infty
\end{equation*}%
for any $\delta >0$ and $r\in (1,1+\alpha )$. Clearly,%
\begin{equation*}
Z(\chi)-j=\sum_{k=1}^{j}\varsigma _{k}(\chi ),
\end{equation*}%
where $\varsigma _{k}(\chi ),k=1,2,...,j$ are independent random variables
each of which is distributed as $\varsigma (\chi ):=Z(\chi )-1$ given $Z(0)=1
$. Since $\mathbf{E}\varsigma (\chi )=0$, it follows by von Bahr-Esseen
inequality (see \cite{BahrEss}) that
\begin{eqnarray*}
\mathbf{E}\left[ \left\vert Z(\chi )-1\right\vert ^{r}|Z(0)=j\right]  &\leq
&2j\mathbf{E}\left[ \left\vert Z(\chi )-1\right\vert ^{r}|Z(0)=1\right]  \\
&\leq &2j(1+\mathbf{E}\left[ Z^{r}(\chi )|Z(0)=1\right] ) \\
&\leq &4j\mathbf{E}\left[ Z^{r}(\chi )|Z(0)=1\right] .
\end{eqnarray*}%
It was show in Lemma 11 of \cite{VFW2008}\ that, for any $r\in
(1,1+\alpha)$ there exists $C=C(r)\in \left( 0,\infty \right) $ such
that, for all $\chi \geq 1$%
\begin{equation*}
\mathbf{E}\left[ Z^{r}(\chi )|Z(0)=1\right] \leq \frac{C}{Q^{r-1}(\chi )}.
\end{equation*}%
As a result%
\begin{equation*}
\mathbf{P}(\left\vert Z(\chi )-j\right\vert >\delta j|Z(0)=j)\leq \frac{4C}{%
\delta ^{r}\left( jQ(\chi )\right) ^{r-1}}.
\end{equation*}%
According to (\ref{Slowly})
\begin{equation*}
jQ(\chi )\geq \varepsilon TQ(\chi )=\varepsilon \frac{1-f_{\chi }(0)}{%
1-f_{\varphi (n)}(0)}\rightarrow \infty.
\end{equation*}%
Hence we conclude that%
\begin{equation*}
\mathbf{P}(\left\vert Z(\chi )-j\right\vert >\delta j|Z(0)=j)\rightarrow 0
\label{Uniform}
\end{equation*}%
uniformly in $j\in \left[ \varepsilon T,NT\right] $. As a result,%
\begin{eqnarray*}
&&\sum_{j=\varepsilon T}^{NT}e^{-\lambda j/T}\mathbf{P}\left( Z(m)=j\right)
\mathbf{P}(0<Z(\chi )<T|Z(0)=j) \\
&&\, =\sum_{j=\varepsilon T}^{NT}e^{-\lambda j/T}\mathbf{P}\left(
Z(m)=j\right) \mathbf{P}(0<Z(\chi )<T,\left\vert Z(\chi )-j\right\vert \leq
\delta j|Z(0)=j) \\
&&\qquad \qquad +o\left( \mathbf{P}\left( 0<Z(m)\leq NT\right) \right) .
\end{eqnarray*}
Now%
\begin{equation*}
\mathbf{P}(0<Z(\chi )<T,j-\delta j\leq Z(\chi )\leq j+\delta j|Z(0)=j)=0
\end{equation*}%
if $T<j(1-\delta )$. Thus,%
\begin{eqnarray*}
&&\sum_{j=\varepsilon T}^{NT}e^{-\lambda j/T}\mathbf{P}\left( Z(m)=j\right)
\mathbf{P}(0<Z(\chi )<T,\left\vert Z(\chi )-j\right\vert \leq \delta
j|Z(0)=j) \\
&&\,=\sum_{j=\varepsilon T}^{T/\left( 1-\delta \right)
}e^{-\lambda j/T}\mathbf{P}\left( Z(m)=j\right) \mathbf{P}(0<Z(\chi
)<T,\left\vert Z(\chi )-j\right\vert \leq \delta j|Z(0)=j).
\end{eqnarray*}
Since%
\begin{equation*}
\mathbf{P}(0<Z(\chi )<T, \left\vert Z(\chi )-j\right\vert \leq \delta
j|Z(0)=j)=\mathbf{P}(\left\vert Z(\chi )-j\right\vert \leq \delta j|Z(0)=j)
\end{equation*}%
if $j\leq T/\left( 1+\delta \right) $ and, by (\ref{Uniform})
\begin{equation*}
\mathbf{P}(\left\vert Z(\chi )-j\right\vert \leq \delta j|Z(0)=j)=1-\mathbf{P%
}(\left\vert Z(\chi )-j\right\vert >\delta j|Z(0)=j)\rightarrow 0
\end{equation*}%
uniformly in $j\in \left[ \varepsilon T,NT\right] $, we may write
\begin{eqnarray*}
&&\sum_{j=\varepsilon T}^{T/\left( 1-\delta \right) }e^{-\lambda
j/T}\mathbf{P}\left( Z(m)=j\right) \mathbf{P}(0<Z(\chi )<T,\,\left\vert Z(\chi
)-j\right\vert \leq \delta j|Z(0)=j) \\
&&\qquad \quad =\sum_{j=\varepsilon T}^{T/\left( 1+\delta \right)
}e^{-\lambda j/T}\mathbf{P}\left( Z(m)=j\right) +o\left( \mathbf{P}\left(
\varepsilon T<Z(m)<T/\left( 1+\delta \right) \right) \right) \\
&&\qquad \qquad +O\left( \mathbf{P}\left( T/\left( 1+\delta \right)
<Z(m)<T/\left( 1-\delta \right) \right) \right) .
\end{eqnarray*}
Since $m\sim n$, we conclude by (\ref{ModifDevoation}) that
\begin{equation*}
\mathbf{P}\left( \varepsilon T<Z(m)<T/\left( 1+\delta \right) \right) \leq C%
\mathbf{P}\left( 0<Z(n)<T\right)   \label{Term21}
\end{equation*}%
and
\begin{equation*}
\lim_{m\sim n\rightarrow \infty }\frac{\mathbf{P}\left( T/\left( 1+\delta
\right) <Z(m)<T/\left( 1-\delta \right) \right) }{\mathbf{P}\left(
0<Z(n)<T\right) }=\frac{1}{\left( 1-\delta \right) ^{\alpha }}-\frac{1}{%
\left( 1+\delta \right) ^{\alpha }}.  \label{Term22}
\end{equation*}%
Further,
\begin{eqnarray*}
\sum_{j=\varepsilon T}^{T/\left( 1+\delta \right) }e^{-\lambda j/T}\mathbf{P}%
\left( Z(m)=j\right)  &=&\sum_{j=\varepsilon T}^{T/\left( 1+\delta \right)
}e^{-\lambda j/T}d\mathbf{P}\left( \varepsilon T\leq Z(m)\leq j\right)  \\
&=&\int_{\varepsilon }^{1/\left( 1+\delta \right) }e^{-\lambda x}d\mathbf{P}%
\left( \varepsilon \leq\frac{Z(m)}{T}\leq x\right) .\qquad \qquad
\end{eqnarray*}%
Hence, using (\ref{Distrib}), (\ref{Laplace0}) and recalling that $m\sim
n\rightarrow \infty $ we get%
\begin{eqnarray*}
&&\int_{\varepsilon }^{1/\left( 1+\delta \right) }e^{-\lambda x}d\mathbf{P}%
\left( \varepsilon <\frac{Z(m)}{T}\leq x\right)   \notag \\
&=&\mathbf{P}\left( 0<Z(m)<T\right) \int_{\varepsilon }^{1/\left( 1+\delta
\right) }e^{-\lambda x}d\mathbf{P}\left( \varepsilon <\frac{Z(m)}{T}\leq
x|0<Z(m)<T\right)   \notag \\
&\sim &\mathbf{P}\left( 0<Z(n)<T\right) \alpha \int_{\varepsilon }^{1/\left(
1+\delta \right) }e^{-\lambda x}x^{\alpha -1}dx.  \label{Term23}
\end{eqnarray*}
Combining the estimates above and letting $\delta $ and $\varepsilon $ to
zero we complete the proof of point 5) of Theorem \ref{T_Population_size}.

\end{document}